\documentclass[11pt]{article}
\usepackage{amssymb,amsmath,amsfonts,amsthm,hyperref,here,enumerate,xfrac,mathtools,graphicx,pstool,xcolor,comment}
\usepackage{framed,tikz}
\usetikzlibrary{arrows}
\usetikzlibrary{arrows.meta}
\definecolor{light}{rgb}{.9,.9,.9}
\definecolor{darkgreen}{rgb}{0,.6,0}
\numberwithin{equation}{section}

\newcommand{\R}{\mathbb R}

\newtheoremstyle{plain}
  {10pt}
  {10pt}
  {\it}
  {0pt}
  {\bf}
  {}
  {\newline}
  {}
\newtheoremstyle{definition}
  {10pt}
  {10pt}
  {}
  {0pt}
  {\bf}
  {}
  {\newline}
  {}
\theoremstyle{plain}
\definecolor{MyDarkBlue}{rgb}{0,0.29,0.7}
\hypersetup{
    colorlinks=true,       
    linkcolor=MyDarkBlue,          
    linkbordercolor=MyDarkBlue,
    citecolor=MyDarkBlue,
    citebordercolor=MyDarkBlue,        
    filecolor=MyDarkBlue,      
    urlcolor=MyDarkBlue           
}

\theoremstyle{plain}

\newtheorem{theorem}{Theorem}[section]
\newtheorem{coro}[theorem]{Corollary}
\newtheorem{lemma}[theorem]{Lemma}
\newtheorem{prop}[theorem]{Proposition}

\theoremstyle{definition}

\newtheorem{definition}[theorem]{Definition}

\begin{document}

\title{Equilibria of plane convex bodies}
\author{Jonas Allemann, Norbert Hungerb\"uhler, Micha Wasem}
\date{\today}
\maketitle
\begin{abstract}
We obtain a formula for the number of horizontal equilibria of a planar convex body $K$ with respect to a center of mass $O$ in terms of the winding number of the evolute of $\partial K$ with respect to $O$. The formula extends to the case where $O$ lies \emph{on} the evolute of $\partial K$ and a suitably modified version holds true for non-horizontal equilibria. \end{abstract}
\section{Introduction}

We study the number of static equilibria of a planar convex body $K$ supported by a horizontal line subject to a uniform vertical gravity field. It is well-known that the number of static equilibria with respect to the centroid of a homogeneous body $K$ is $\geq 4$ (see~\cite{Domokos1994} and Proposition~\ref{number_equilibria} below). It was pointed out in~\cite{Varkonyi2006}, that this result is equivalent to the Four-vertex Theorem. For an arbitrary center of mass, one can find planar convex bodies with only one stable and one unstable equilibrium -- the 3-dimensional counterparts of such objects are known as \emph{roly-poly toys}. In~\cite{Varkonyi2006} it is shown that there exists a homogeneous convex roly-poly toy with exactly one stable and one unstable equilibrium, the so-called g\"omb\"oc -- thus answering a long-standing  conjecture by Arnol'd in the affirmative.

In this article, we provide a geometric characterisation of the number  $n$ of static equilibria of a planar convex body $K$ in terms of the winding number of the evolute of $\partial K$ with respect to a given center of mass $O$ of $K$: If $\partial K$ is parametrized by a positively oriented curve $\gamma$ and $O$ is not a point of the evolute of $\partial K$, then the winding number of the evolute of $\partial K$ is an integer $m\leq 0$ and the formula
\begin{equation}\label{main-formula}
n=2-2m
\end{equation}
holds true. We will show that this formula remains valid, if $O$ is a point of the evolute of $\partial K$, possibly even a cusp, but in this case, $m$ might be half-integer valued. Our main theorem is the following:

\begin{theorem}\label{main_arbitrary_zeros}
Let $K$ be a strongly convex compact set with $C^3$-boundary $\partial K$ such that the curvature of $\partial K$ has only finitely many stationary points, and let $O$ be a point
in the plane. Then 
the number $n$ of horizontal equilibria of $K$ with respect to $O$ is given by
$$
n=2-2m,
$$
where $0\geq m \in \frac12\mathbb Z$ is the winding number of the evolute of $\partial K$ with respect to $O$.
\end{theorem}

The strategy of the proof is to identify the horizontal equilibria as zeros of the first derivative of a \emph{support function} that parametrizes $\partial K$ and using the zero-counting integral developed in~\cite{zeros} in order to count its zeros. The resulting integral can then be related to the generalized winding number (see~\cite{residue}) of the evolute of $\partial K$.

In Section~\ref{sectionoblique} we replace the horizontal supporting line of the body $K$ by an inclined line with inclination angle $\alpha\in(-\frac\pi2,\frac\pi2)$. It is interesting that for any angle $\alpha\ne 0$, there exist homogeneous bodies $K$ such that the inequality $n\geq 4$ fails. In fact, for every $\alpha\ne 0$, there are such bodies with exactly one metastable equilibrium and also bodies with exactly one stable and one unstable equilibrium with respect to the centroid (see Proposition~\ref{obliqueprop} below). Furthermore, a formula like~\eqref{main-formula} holds true for $\alpha\ne 0$, where $m$ is the winding number of the evolute of a suitable modification of $\partial K$.

\section{Support Functions}
For $x,y\in\R^2$, let $(x,y) = \left\{tx + (1-t)y, t\in(0,1)\right\}$ denote the line segment between the points $x$ and $y$. A set $K\subset\R^2$ is called \emph{convex} if for any $x,y\in K$ it holds that $(x,y)\cap K=(x,y)$. The set $K$ is called \emph{strictly convex} if $(x,y)\cap\mathring K=(x,y)$ for any $x,y\in K$. A bounded convex set $K\subset \R^2$ with $C^n$-boundary, $n\geq2$ is called \emph{strongly convex} if $\partial K$ can be parametrized by a curve $\gamma:S^1\to\partial K$ such that $\|\dot \gamma\|=1$ and $\ddot\gamma$ does not vanish. We will use the identification $S^1\cong \mathbb R/2\pi\mathbb Z$ and the notation $u(\varphi)=(\cos (\varphi),\sin(\varphi))^\top$ throughout this article.

The boundary $\partial K$ of a strictly convex compact set $K$ admits a parametrization by \emph{support functions} $p$ and $q$, i.e., there exists a parametrization $z:S^1\to \partial K$ such that $z(\varphi) = p(\varphi)u(\varphi)+q(\varphi)u'(\varphi)$ (see~\cite{euler1778}), as indicated in Figure~\ref{fig-support}: Here
$S$ is a reference point and $\ell$ a ray emanating in $S$ from which we measure angles.

\begin{figure}[h!]
\begin{center}
\definecolor{blue}{rgb}{0.,0.,1.}
\definecolor{red}{rgb}{1.,0.,0.}
\definecolor{ccwwqq}{rgb}{0.8,0.4,0.}
\begin{tikzpicture}[line cap=round,line join=round,x=0.8cm,y=0.8cm]
\clip(-4.16,-3.46) rectangle (5.74,3.46);
\draw [->,shift={(1.9591469216489055,-1.6439422743241814)}] (0,0) -- (0.:0.9) arc (0.:55.26533603178533:0.9);
\draw [shift={(3.8256876429761633,1.0482033910159714)}] (0,0) -- (-124.73466396821468:0.4) arc (-124.73466396821468:-34.73466396821468:0.4);
\draw[color=ccwwqq,fill=ccwwqq,fill opacity=0.25, smooth,samples=100,domain=0.0:6.283185307179586] plot ({3.8*cos(deg(\x))-0.3*sin(3.0*deg(\x))},{2.8*sin(deg(\x))+0.2*cos(2.0*deg(\x))});
\draw [domain=-4.16:5.74] plot(\x,{(-7.786461130023362--1.4588134273615847*\x)/-2.1040731660110015});
\draw [line width=.8pt,domain=1.9591469216489055:5.740000000000002] plot(\x,{(-5.360654225861506-0.*\x)/3.2608530783510945});
\fill(3.867618685708059,0.816675618481608) circle (0.02);
\draw [line width=1pt,color=red] (3.8256876429761633,1.0482033910159714)-- (1.9591469216489055,-1.6439422743241814);
\draw [line width=1pt,color=blue] (3.8256876429761633,1.0482033910159714)-- (2.3485560108715426,2.072340523556868);
\draw [-{>[scale=2,
          length=2,
          width=3]},color=darkgreen] (1.9591469216489055-.016,-1.6439422743241814+.016) -- (2.891009427511566-.016,-0.29990001626549634+.016);
\draw [-{>[scale=2,
          length=2,
          width=3]},color=orange] (1.9591469216489055,-1.6439422743241814) -- (.615,-.712);
\begin{footnotesize}
\draw [fill=black] (2.3485560108715426,2.072340523556868) circle (1.0pt);
\draw[color=black] (2.46,2.36) node {$Z$};
\draw[color=black] (5.58,-1.89) node {$\ell$};
\draw [fill=black] (3.8256876429761633,1.0482033910159714) circle (1pt);
\draw [fill=black] (1.9591469216489055,-1.6439422743241814) circle (1.0pt);
\draw[color=black] (1.7,-1.79) node {$S$};
\draw[color=black] (2.5,-1.4) node {$\varphi$};
\draw[color=red] (2.5,0.01) node {$p(\varphi)$};
\draw[color=blue] (3.5,1.9) node {$q(\varphi)$};
\draw[color=black] (2,-0.71) node {$\textcolor{darkgreen}{u(\varphi)}$};
\draw (.515,-1.2) node {$\textcolor{orange}{u'(\varphi)}$};
\draw (-1,1.2) node {$\textcolor{ccwwqq}{K}$};
\draw (-3,-2) node {$\textcolor{ccwwqq}{\partial K}$};
\end{footnotesize}
\end{tikzpicture}
\caption{The support functions $p$ and $q$ of a strictly convex compact set $K$.}\label{fig-support}
\end{center}
\end{figure}
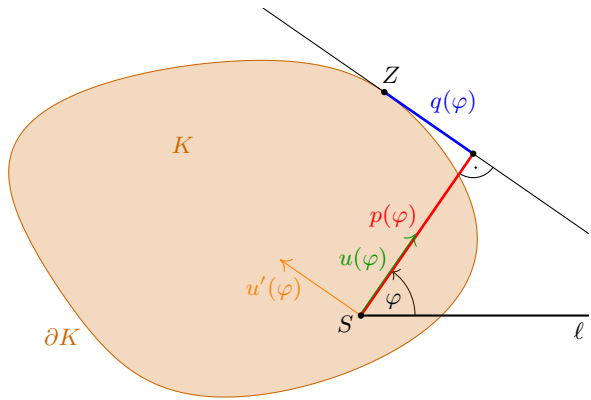
In fact, for fixed $\varphi$, the orthogonal projection of $K$ to the line $g=\{\lambda u(\varphi)\mid \lambda\in\mathbb R\}$
is a compact interval (see Figure~\ref{fig-support2}), and for its endpoint $P$ we have
$$
P=pu(\varphi)\text{ with $p=\max\{ \langle X,u(\varphi)\rangle\mid X\in K \}$}.
$$
Since $K$ is strictly convex, $p=\langle Z,u(\varphi)\rangle$ for a unique $Z\in K$. Hence, by choosing $p(\varphi)=p$, and
$q(\varphi)$ as the oriented distance of $Z$ and $P$ we have indeed
$$
Z=z(\varphi)=p(\varphi)u(\varphi)+q(\varphi)u'(\varphi).
$$
\begin{figure}[h!]
\begin{center}
\definecolor{ccwwqq}{rgb}{0.8,0.4,0.}
\begin{tikzpicture}[line cap=round,line join=round,x=.8cm,y=.8cm]
\clip(-4.16,-3.46) rectangle (5.74,3.46);
\draw [->,shift={(1.9591469216489055,-1.6439422743241814)}] (0,0) -- (0.:0.9) arc (0.:55.26533603178533:0.9);
\draw [shift={(3.8256876429761633,1.0482033910159714)}] (0,0) -- (-124.73466396821468:0.4) arc (-124.73466396821468:-34.73466396821468:0.4);
\draw[color=ccwwqq,fill=ccwwqq,fill opacity=0.25, smooth,samples=100,domain=0.0:6.283185307179586] plot ({3.8*cos(deg(\x))-0.3*sin(3.0*deg(\x))},{2.8*sin(deg(\x))+0.2*cos(2.0*deg(\x))});
\draw [domain=-4.16:5.74] plot(\x,{(-7.786461130023362--1.4588134273615847*\x)/-2.1040731660110015});
\draw [line width=.8pt,domain=1.9591469216489055:5.740000000000002] plot(\x,{(-5.360654225861506-0.*\x)/3.2608530783510945});
\fill(3.867618685708059,0.816675618481608) circle (0.02);
\fill[line width=.8pt] (3.867618685708059,0.816675618481608) circle (0.02);
\draw [-{>[scale=2,
          length=2,
          width=3]},color=darkgreen] (1.9591469216489055-.016,-1.6439422743241814+.016) -- (2.891009427511566-.016,-0.29990001626549634+.016);
\draw [-{>[scale=2,
          length=2,
          width=3]},color=orange] (1.9591469216489055,-1.6439422743241814) -- (.615,-.712);
\draw [domain=-4.16:5.74] plot(\x,{(--8.342794091419034-2.6921456653401528*\x)/-1.8665407213272578});
\begin{scriptsize}
\draw [fill=black] (2.3485560108715426,2.072340523556868) circle (1.0pt);
\draw[color=black] (2.46,2.3) node {$Z$};
\draw [fill=black] (3.8256876429761633,1.0482033910159714) circle (1.0pt);
\draw[color=black] (3.82,1.37) node {$P$};
\draw [fill=black] (1.9591469216489055,-1.6439422743241814) circle (1.0pt);
\draw[color=black] (2.5,-1.4) node {$\varphi$};
\draw[color=black] (2,-0.71) node {$\textcolor{darkgreen}{u(\varphi)}$};
\draw (.515,-1.2) node {$\textcolor{orange}{u'(\varphi)}$};
\draw[color=black] (5.02,2.37) node {$g$};
\draw (-1,1.2) node {$\textcolor{ccwwqq}{K}$};
\draw (-3,-2) node {$\textcolor{ccwwqq}{\partial K}$};
\draw[color=black] (5.58,-1.89) node {$\ell$};
\draw[color=black] (1.7,-1.75) node {$S$};
\end{scriptsize}
\end{tikzpicture}
\caption{Existence and uniqueness of the support functions.}\label{fig-support2}
\end{center}
\end{figure}
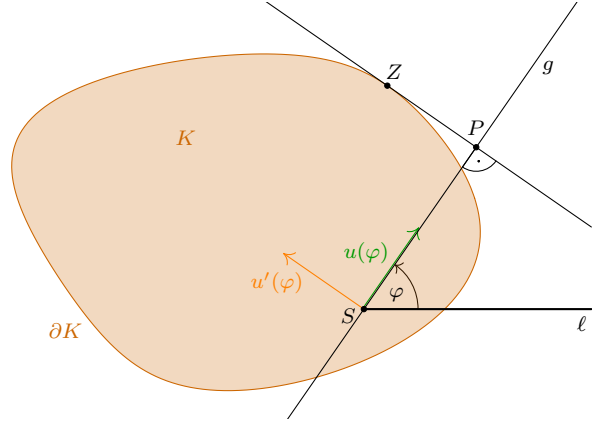The connection between the regularity of the boundary curve $\partial K$ and the support 
functions is described in the following Lemma. Note that here we need that $K$ is strongly convex.
\begin{lemma}\label{regularity}
Let $K$ be a strongly convex compact set with $C^n$ boundary $\partial K$, $n\geq 2$. Then $\partial K$ can be parametrized by 
$\varphi\mapsto z(\varphi)=p(\varphi)u(\varphi)+p'(\varphi)u'(\varphi)$, where $p\in C^n(S^1,\R^2)$.
\end{lemma}
This result is remarkable in that $p$ as a function of arc length $s$ along $\partial K$ instead of $\varphi$ 
is only in $C^{n-1}$ in general.
\begin{proof}
Let $\gamma$ be a $C^n$ arc-length parametrization of $\partial K$ and let $J=\left(\begin{smallmatrix}0&-1\\1&0\end{smallmatrix}\right)$. Observe that $\{-J\dot \gamma(s),\dot\gamma(s)\}$ forms an orthonormal basis of $\R^2$ for every $s$, where
the dot indicates the derivative with respect to arc length. Hence we may write 
\begin{equation}\label{eq-s}
\gamma(s)=-\textcolor{red}{p(s)}J\dot\gamma(s)+\textcolor{blue}{q(s)}\dot\gamma(s),
\end{equation} 
where
\begin{alignat*}{2}
\textcolor{red}{p(s)}&=-\langle\gamma(s),J\dot\gamma(s)\rangle&&\in C^{n-1}\\
\textcolor{blue}{q(s)}&=\textcolor{white}{-}\langle\gamma(s),\dot\gamma(s)\rangle&&\in C^{n-1}.
\end{alignat*}
See Figure~\ref{fig-regularity}.
\begin{figure}[h!]
\begin{center}
\definecolor{blue}{rgb}{0.,0.,1.}
\definecolor{red}{rgb}{1.,0.,0.}
\definecolor{ccwwqq}{rgb}{0.8,0.4,0.}
\begin{tikzpicture}[line cap=round,line join=round,x=.8cm,y=.8cm]
\clip(-4.16,-3.46) rectangle (5.74,3.46);
\draw [shift={(1.9591469216489055,-1.6439422743241814)}] (0,0) -- (0.:1.2) arc (0.:55.26533603178533:1.2) -- cycle;
\draw [shift={(3.8256876429761633,1.0482033910159714)}] (0,0) -- (-124.73466396821468:0.4) arc (-124.73466396821468:-34.73466396821468:0.4) -- cycle;
\draw[color=ccwwqq,fill=ccwwqq,fill opacity=0.25, smooth,samples=100,domain=0.0:6.283185307179586] plot ({3.8*cos(deg(\x))-0.3*sin(3.0*deg(\x))},{2.8*sin(deg(\x))+0.2*cos(2.0*deg(\x))});
\draw [domain=-4.16:5.74] plot(\x,{(-7.786461130023362--1.4588134273615847*\x)/-2.1040731660110015});
\draw [line width=1pt,domain=1.9591469216489055:5.740000000000002] plot(\x,{(-5.360654225861506-0.*\x)/3.2608530783510945});
\fill(3.867618685708059,0.816675618481608) circle (0.02);
\draw [line width=1pt,color=red] (3.8256876429761633,1.0482033910159714)-- (1.9591469216489055,-1.6439422743241814);
\draw [line width=1pt,color=blue] (3.8256876429761633,1.0482033910159714)-- (2.3485560108715426,2.072340523556868);
\draw [-{>[scale=2.5,
          length=2,
          width=3]},color=darkgreen] (1.9591469216489055,-1.6439422743241814) -- (2.891009427511566,-0.2999000162654961);
\draw [->,line width=.75pt] (2.3485560108715426,2.072340523556868) -- (1.4166935050088822,0.7282982654981829);
\draw [->,line width=.75pt] (2.3485560108715426,2.072340523556868) -- (1.0045137528128574,3.0042030294195285);
\begin{footnotesize}
\draw [fill=black] (2.3485560108715426,2.072340523556868) circle (1.0pt);
\draw[color=black] (2.8,2.27) node {$\gamma(s)$};
\draw [fill=black] (3.8256876429761633,1.0482033910159714) circle (1pt);
\draw [fill=black] (1.9591469216489055,-1.6439422743241814) circle (1.0pt);
\draw[color=black] (2.66,-1.33) node {$\varphi(s)$};
\draw[color=red] (2.6,0.1) node {$p(s)$};
\draw[color=blue] (3.42,1.81) node {$q(s)$};
\draw[color=black] (1.94,-0.91) node {\textcolor{darkgreen}{$u(s)$}};
\draw[color=black] (1.26,1.47) node {$J\dot\gamma(s)$};
\draw[color=black] (2.12,2.7) node {$\dot\gamma(s)$};
\draw (-1,1.2) node {$\textcolor{ccwwqq}{K}$};
\draw[color=black] (5.58,-1.89) node {$\ell$};
\draw[color=black] (1.7,-1.75) node {$S$};
\end{footnotesize}
\end{tikzpicture}
\caption{Parametrisation by arc length.}\label{fig-regularity}
\end{center}
\end{figure}
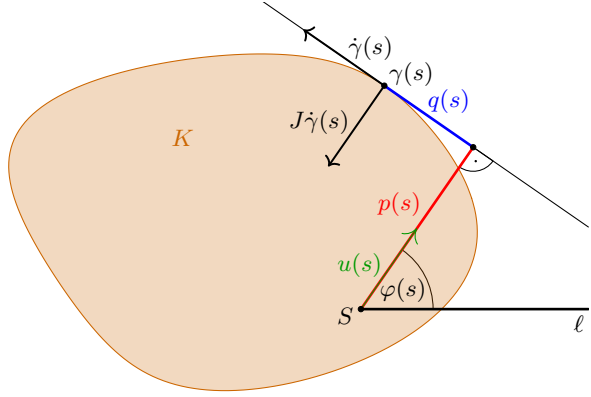
It holds that $\varphi(s) = \arg \dot \gamma(s) - \frac\pi 2=-\arctan\left(\frac{\dot\gamma_1(s)}{\dot\gamma_2(s)}\right)$ is of class $C^{n-1}$. Hence $\varphi\mapsto q(s(\varphi))$ is of class $C^{n-1}$ and
\begin{equation}\label{phi}
\frac{\mathrm d\varphi}{\mathrm ds} = - \frac{1}{1+\left(\frac{\dot \gamma_1}{\dot\gamma_2}\right)^2} \cdot \frac{\ddot\gamma_1\dot\gamma_2-\dot\gamma_1\ddot\gamma_2}{\dot \gamma_2^2} = -\frac{\ddot\gamma_1\dot\gamma_2-\dot\gamma_1\ddot\gamma_2}{\dot\gamma_1^2+\dot\gamma_2^2} = -\langle \dot \gamma,J\ddot\gamma\rangle.
\end{equation}
We now show that the derivative $(p\circ s)'(\varphi) = (q\circ s)(\varphi)$ which implies that $\varphi \mapsto p(s(\varphi))$ is of class $C^n$: Indeed we have
\begin{eqnarray*}
(p\circ s)'(\varphi) & = & -\frac{\mathrm d}{\mathrm d\varphi}\Big\langle\gamma(s(\varphi)),J\dot\gamma(s(\varphi))\Big\rangle \\
& = & -\Big\langle \dot\gamma(s(\varphi)),J\dot\gamma(s(\varphi))\Big\rangle \frac{\mathrm ds}{\mathrm d\varphi} - \Big\langle \gamma(s(\varphi)),J\ddot\gamma(s(\varphi))\Big\rangle \frac{\mathrm ds}{\mathrm d\varphi}\\
& \stackrel{\eqref{phi}}{=}& \frac{\langle \gamma(s(\varphi)),J\ddot\gamma(s(\varphi))\rangle}{\langle\dot\gamma(s(\varphi)),J\ddot\gamma(s(\varphi))\rangle} \\ & = & \Big\langle \gamma(s(\varphi)),\dot\gamma(s(\varphi))\Big\rangle = (q\circ s)(\varphi),
\end{eqnarray*}
where we have used $J\ddot\gamma \parallel \dot\gamma$ in the last line.
\end{proof}
\begin{coro}
Let $K$ be a strongly convex compact set with $C^n$ boundary $\partial K$, $n\geq 2$. If $\partial K$ 
is parametrized by $z(\varphi) = p(\varphi)u(\varphi) + p'(\varphi)u'(\varphi)$, then $z'(\varphi)=u'(\varphi)\rho(\varphi)$, where $\rho(\varphi) = p(\varphi) + p''(\varphi)$ is
the radius of curvature of $\partial K$ in $z(\varphi)$.
\end{coro}
{\em Proof.}
It follows from Lemma \ref{regularity} that $p$ is of class $C^n$. 
First, by direct calculation, we see that $z'=(p+p'')u'$, because $u''=-u$.
If $n\geq 3$, we compute $z'' = (p + p'')u'' + (p'+p''')u'$. 
Since the radius of curvature $\rho$ is the projection of $z''$ onto $u''$ we obtain the desired result. 
If $n=2$ we consider again a parametrization $\gamma$ of $\partial K$ by arc length 
and use $\rho \ddot \gamma = J\dot \gamma$ and $J^2=-\operatorname{id}$ to compute by~(\ref{eq-s})
$$\begin{aligned}
\dot \gamma & = -pJ\ddot \gamma  - \dot pJ\dot\gamma  + q\ddot \gamma  + \dot q\dot\gamma \\
& = \frac{p}{\rho}\dot\gamma  - \dot pJ\dot\gamma + \frac{q}{\rho}J\dot\gamma  + \dot q\dot\gamma\\
& = \left(\frac{p}{\rho}+\dot q\right)\dot \gamma + \left(\frac{q}{\rho}-\dot p\right)J\dot\gamma.
\end{aligned}$$
Hence we have $\frac{p}{\rho}+\dot q\equiv1$ and $\frac{q}{\rho}-\dot p\equiv0$. Using $\dot q = 1 - \frac{p}{\rho}$ we find
\makeatletter
\renewcommand\tagform@[1]{\maketag@@@{\ignorespaces#1\unskip\@@italiccorr}}
\makeatother
\begin{equation}
\begin{aligned}
p + p''  & = p + q' = p + \dot q \cdot \frac{\mathrm ds}{\mathrm d\varphi} = p+\frac{1-\frac{p}{\rho}}{-\langle\dot\gamma,J\ddot\gamma\rangle} = p+\frac{1-\frac{p}{\rho}}{\frac{1}{\rho}\langle\dot\gamma,\dot\gamma\rangle} = \rho. 
\end{aligned}\tag{$\Box$}
\end{equation}
\makeatletter
\renewcommand\tagform@[1]{\maketag@@@{\ignorespaces(#1)\unskip\@@italiccorr}}
\makeatother

We will now collect a few expressions for relevant geometric quantities in terms of the parametrization for $\partial K$ 
from Lemma~\ref{regularity}: First of all the arc length $s(\varphi)$ of $z|_{[0,\varphi]}$ is given by
\begin{equation}\label{eq-arc}
s(\varphi) = \int_0^\varphi |z'|\,\mathrm d\phi = \int_0^\varphi (p+p'')\,\mathrm d\phi = \int_0^\varphi p\,\mathrm d\phi + p'(\varphi)-p'(0)
\end{equation}
and hence the perimeter of $K$ is
$$s(2\pi) = \int_0^{2\pi} p\,\mathrm d\varphi =:L.$$
The center of mass of the curve $\partial K$ is given by
$$
\frac{1}{L}\int_0^{2\pi} z |z'|\,\mathrm d\varphi = \frac{1}{L}\int_0^{2\pi}(pu+p'u')(p+p'')\,\mathrm d\varphi = \frac{1}{L}\int_0^{2\pi}\left(p^2-\frac{p'^2}{2}\right)u\,\mathrm d\varphi,
$$
where we have integrated by parts. Similarly, the area $A$ of $K$ is given by
$$
A=\frac{1}{2}\int_0^{2\pi}(p^2-p'^2)\,\mathrm d\varphi,
$$
and the centroid $O$ of $K$ by
$$
O=\frac{1}{3A}\int_0^{2\pi}(pu+p'u')p(p+p'')\,\mathrm d\varphi.
$$
\section{The evolute of \boldmath{$\partial K$}}
Let $K$ be a strongly convex set of class $C^2$. Then,
the evolute of $\partial K$  is given by
\begin{eqnarray}
e(\varphi)&=&z(\varphi)-\rho(\varphi)u(\varphi)\notag\\
&=& z(\varphi) + (p(\varphi)+p''(\varphi))u''(\varphi)\notag \\
&=& p(\varphi)u(\varphi) + p'(\varphi)u'(\varphi)-(p(\varphi)+p''(\varphi))u(\varphi) \notag\\
&=& p'(\varphi)u'(\varphi)-p''(\varphi)u(\varphi).\label{eq-evolute}
\end{eqnarray} 
Thus, the evolute is obtained from the original curve $\partial K$ by 
replacing its support function $p$ by $p'$ and a rotation about $90^\circ$.
Formula~(\ref{eq-evolute}) shows that all \emph{parallel curves} of $\partial K$, which have support function 
$p + \text{constant}$, have the same evolute as $K$.

\subsection{Curves of Constant Width}\label{const-width}
Suppose that $K$ is a strongly convex set with $C^2$-boundary $\partial K$ and assume in addition that $\partial K$ is a curve of constant width 
$d>0$. Then $\partial K$ can be parametrized by a support function $p$ that satisfies $p(\varphi)+p(\varphi + \pi)\equiv d$. This equation implies that $p^{(k)}(\varphi)=-p^{(k)}(\varphi+\pi)$ for $k=1,2$ and it follows for the evolute $e$ of $\partial K$
$$
\begin{aligned}
e(\varphi+\pi) & = p'(\varphi+\pi)\overbrace{u'(\varphi+\pi)}^{=-u'(\varphi)}-p''(\varphi+\pi)\overbrace{u(\varphi+\pi)}^{=-u(\varphi)}\\
& = p'(\varphi)u'(\varphi)-p''(\varphi)u(\varphi)\\
& = e(\varphi).
\end{aligned}
$$
Hence $e:S^1\to\R^2$ is $\pi$-periodic. This means that the evolute 
of a curve of constant width is traversed twice.
\subsection{Cusps of the evolute}
Even if $\partial K$ is a smooth regular curve, its evolute
has necessarily at least four singular points (cusps). The situation is described in the following lemma:
\begin{lemma}\label{cuspsofevolute}
Let $K$ be strongly convex and compact with $\partial K$ of class $C^3$
parametrized by $\varphi\mapsto z(\varphi)=p(\varphi)u(\varphi)+p'(\varphi)u'(\varphi)$. We assume, that
the curvature of $\partial K$ has only finitely many stationary points.
Then, the evolute of $\partial K$, given by $\varphi\mapsto e(\varphi)=p'(\varphi)u'(\varphi)-p''(\varphi)u(\varphi)$,
is regular and of class $C^2$ except for points where the radius of curvature $\rho$ of $\partial K$ is stationary. 
More precisely: 
\begin{itemize}
\item If $\rho$ has a local minimum in $\varphi_0$, then $e$ has a cusp in $\varphi_0$
pointing towards the point $z(\varphi_0)$ (see Figure~\ref{fig-evolute}).
\item If $\rho$ has a local maximum in $\varphi_0$, then $e$ has a cusp in $\varphi_0$
pointing away from  the point $z(\varphi_0)$  (see Figure~\ref{fig-evolute}).
\item If $\rho$ has a saddle point in $\varphi_0$, then $e$ is $C^1$ in $\varphi_0$.
\end{itemize}
\end{lemma}
{\em Remarks.} \begin{itemize}
\item By the Four-vertex Theorem (see~\cite{four}, \cite{kneser} or~\cite{osserman}), it follows that the evolute of $\partial K$
has at least four cusps. Since maxima and minima alternate, the number of cusps is
always even.
\item Note that the $C^2$-regularity of $e$ is not evident, since the parametrization of $e$
with respect to $\varphi$ is obviously only $C^1$ in general.
\item The connection between the cusps of the evolute and strict local extrema of the
base curve has first been observed by G.\,H. Light~\cite{light}.
\end{itemize}
\begin{proof}
First of all, note that
$$
z' = \rho u' \text{\quad and\quad } e'= -\rho'u
$$
which shows that $\langle z',e'\rangle = 0$. Suppose now, that $\rho'(\varphi_0)=0$.
\begin{itemize}
\item[1.~case:] $\rho$ has a local minimum in $\varphi_0$. Then
$$
\lim_{\varphi\nearrow\varphi_0}\frac{e'(\varphi)}{\Vert e'(\varphi)\Vert}=u(\varphi_0) \text{\quad and\quad }
\lim_{\varphi\searrow\varphi_0}\frac{e'(\varphi)}{\Vert e'(\varphi)\Vert}=-u(\varphi_0).
$$
\item[2.~case:] $\rho$ has a local maximum in $\varphi_0$. Then
$$
\lim_{\varphi\nearrow\varphi_0}\frac{e'(\varphi)}{\Vert e'(\varphi)\Vert}=-u(\varphi_0) \text{\quad and\quad }
\lim_{\varphi\searrow\varphi_0}\frac{e'(\varphi)}{\Vert e'(\varphi)\Vert}=u(\varphi_0).
$$
\item[3.~case:] $\rho$ has a saddle point in $\varphi_0$, i.e., $\rho'$ does not change sign in $\varphi_0$. Then
$\lim_{\varphi\to\varphi_0}\frac{e'(\varphi)}{\Vert e'(\varphi)\Vert}$ exists, and $e$ is $C^1$ in $\varphi_0$.
\end{itemize}
To check the regularity of the evolute, we interpret the curve locally as a graph of a function $x_2(x_1)$
or $x_1(x_2)$. Then, by the chain rule, we have for $x_1(x_2)$
$$
\frac{\mathrm dx_2}{\mathrm dx_1}(\varphi)=\frac{\frac{\mathrm dz_2(\varphi)}{\mathrm d\varphi}}{\frac{\mathrm dz_1(\varphi)}{\mathrm d\varphi}}=\tan(\varphi)
$$
and
$$
\frac{\mathrm d^2x_2}{\mathrm dx_1^2}(\varphi)=\frac{\frac{\mathrm d\tan(\varphi)}{\mathrm d\varphi}}{\frac{\mathrm dz_1(\varphi)}{\mathrm d\varphi}}=-\frac1{\rho'(\varphi)\cos^3(\varphi)}.
$$
The case $x_1(x_2)$ is similar. Since $\rho'$ is $C^0$, we conclude that locally, in points $\varphi$ where $\rho'(\varphi)\neq 0$,
the curve $e$ is $C^2$.
\end{proof}
\begin{coro}
If we count the arc length of the evolute $e$ between two cusps alternating positive and negative, the resulting sum vanishes  (see Figure~\ref{fig-evolute}).
\end{coro}
{\em Proof.}
The factor $\rho '$ in $e'= -\rho'u$ changes its sign in every cusp. The length of $e$ is
$$
\int_0^{2\pi}\Vert e'\Vert\, \mathrm d\varphi = \int_0^{2\pi}|\rho'|\, \mathrm d\varphi
$$
and hence the alternating sum of the lengths between cusps equals
\makeatletter
\renewcommand\tagform@[1]{\maketag@@@{\ignorespaces#1\unskip\@@italiccorr}}
\makeatother
\begin{equation}
\int_0^{2\pi}\rho'\, \mathrm d\varphi =0.\tag{$\Box$}
\end{equation}
\makeatletter
\renewcommand\tagform@[1]{\maketag@@@{\ignorespaces(#1)\unskip\@@italiccorr}}
\makeatother
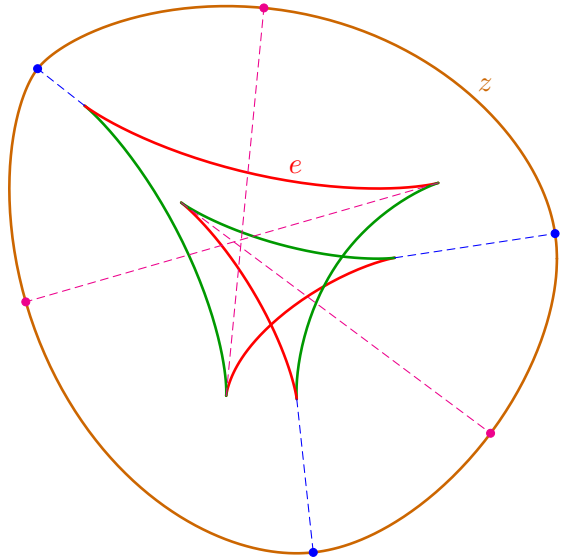
\begin{figure}[h!]
\begin{center}
\definecolor{ccwwqq}{rgb}{0.8,0.4,0.}
\begin{tikzpicture}[line cap=round,line join=round,x=2.8cm,y=2.8cm]
\draw[line width=1pt,color=ccwwqq,smooth,samples=100,domain=0.0:6.283185307179586] plot ({cos(deg(\x))*(1.3 + 1/30*cos(3*deg(\x)) + 1/133*cos(5*deg(\x)) - 1/12*sin(2*deg(\x)) +1/15*sin(3*deg(\x))) -  sin(deg(\x))*(-(1/6)*cos(2*deg(\x)) + 1/5*cos(3*deg(\x)) - 1/10*sin(3*deg(\x)) - 5/133*sin(5*deg(\x)))},{sin(deg(\x))*(1.3 + 1/30*cos(3*deg(\x)) + 1/133*cos(5*deg(\x)) - 1/12*sin(2*deg(\x)) + 1/15*sin(3*deg(\x))) + cos(deg(\x))*(-(1/6)*cos(2*deg(\x)])+ 1/5*cos(3*deg(\x)) - 1/10*sin(3*deg(\x)) -  5/133*sin(5*deg(\x)))});
\draw[line width=1pt,color=red,smooth,samples=20,domain=0.1500085849:1.473877184] plot ({-cos(deg(\x))*(-(3/10)*cos(3*deg(\x)) - 25/133*cos(5*deg(\x)) + 1/3*sin(2*deg(\x)) -3/5*sin(3*deg(\x))) - sin(deg(\x))*(-(1/6)*cos(2*deg(\x)) + 1/5*cos(3*deg(\x)) - 1/10*sin(3*deg(\x)) - 5/133*sin(5*deg(\x)))},{-sin(deg(\x))*(-(3/10)*cos(3*deg(\x)) - 25/133 *cos(5*deg(\x)) + 1/3*sin(2*deg(\x)) - 3/5*sin(3*deg(\x))) + cos(deg(\x))*(-(1/6)*cos(2*deg(\x)) + 1/5*cos(3*deg(\x)) - 1/10*sin(3*deg(\x)) - 5/133*sin(5*deg(\x)))});

\draw[line width=1pt,color=darkgreen,smooth,samples=20,domain=1.473877184:2.474249512] plot ({-cos(deg(\x))*(-(3/10)*cos(3*deg(\x)) - 25/133*cos(5*deg(\x)) + 1/3*sin(2*deg(\x)) -3/5*sin(3*deg(\x))) - sin(deg(\x))*(-(1/6)*cos(2*deg(\x)) + 1/5*cos(3*deg(\x)) - 1/10*sin(3*deg(\x)) - 5/133*sin(5*deg(\x)))},{-sin(deg(\x))*(-(3/10)*cos(3*deg(\x)) - 25/133 *cos(5*deg(\x)) + 1/3*sin(2*deg(\x)) - 3/5*sin(3*deg(\x))) + cos(deg(\x))*(-(1/6)*cos(2*deg(\x)) + 1/5*cos(3*deg(\x)) - 1/10*sin(3*deg(\x)) - 5/133*sin(5*deg(\x)))});

\draw[line width=1pt,color=red,smooth,samples=20,domain=2.474249512:3.422632918] plot ({-cos(deg(\x))*(-(3/10)*cos(3*deg(\x)) - 25/133*cos(5*deg(\x)) + 1/3*sin(2*deg(\x)) -3/5*sin(3*deg(\x))) - sin(deg(\x))*(-(1/6)*cos(2*deg(\x)) + 1/5*cos(3*deg(\x)) - 1/10*sin(3*deg(\x)) - 5/133*sin(5*deg(\x)))},{-sin(deg(\x))*(-(3/10)*cos(3*deg(\x)) - 25/133 *cos(5*deg(\x)) + 1/3*sin(2*deg(\x)) - 3/5*sin(3*deg(\x))) + cos(deg(\x))*(-(1/6)*cos(2*deg(\x)) + 1/5*cos(3*deg(\x)) - 1/10*sin(3*deg(\x)) - 5/133*sin(5*deg(\x)))});

\draw[line width=1pt,color=darkgreen,smooth,samples=20,domain=3.422632918:4.819518211] plot ({-cos(deg(\x))*(-(3/10)*cos(3*deg(\x)) - 25/133*cos(5*deg(\x)) + 1/3*sin(2*deg(\x)) -3/5*sin(3*deg(\x))) - sin(deg(\x))*(-(1/6)*cos(2*deg(\x)) + 1/5*cos(3*deg(\x)) - 1/10*sin(3*deg(\x)) - 5/133*sin(5*deg(\x)))},{-sin(deg(\x))*(-(3/10)*cos(3*deg(\x)) - 25/133 *cos(5*deg(\x)) + 1/3*sin(2*deg(\x)) - 3/5*sin(3*deg(\x))) + cos(deg(\x))*(-(1/6)*cos(2*deg(\x)) + 1/5*cos(3*deg(\x)) - 1/10*sin(3*deg(\x)) - 5/133*sin(5*deg(\x)))});

\draw[line width=1pt,color=red,smooth,samples=20,domain=4.819518211:5.642255832] plot ({-cos(deg(\x))*(-(3/10)*cos(3*deg(\x)) - 25/133*cos(5*deg(\x)) + 1/3*sin(2*deg(\x)) -3/5*sin(3*deg(\x))) - sin(deg(\x))*(-(1/6)*cos(2*deg(\x)) + 1/5*cos(3*deg(\x)) - 1/10*sin(3*deg(\x)) - 5/133*sin(5*deg(\x)))},{-sin(deg(\x))*(-(3/10)*cos(3*deg(\x)) - 25/133 *cos(5*deg(\x)) + 1/3*sin(2*deg(\x)) - 3/5*sin(3*deg(\x))) + cos(deg(\x))*(-(1/6)*cos(2*deg(\x)) + 1/5*cos(3*deg(\x)) - 1/10*sin(3*deg(\x)) - 5/133*sin(5*deg(\x)))});

\draw[line width=1pt,color=darkgreen,smooth,samples=20,domain=5.642255832:0.1500085849+6.28318530717958] plot ({-cos(deg(\x))*(-(3/10)*cos(3*deg(\x)) - 25/133*cos(5*deg(\x)) + 1/3*sin(2*deg(\x)) -3/5*sin(3*deg(\x))) - sin(deg(\x))*(-(1/6)*cos(2*deg(\x)) + 1/5*cos(3*deg(\x)) - 1/10*sin(3*deg(\x)) - 5/133*sin(5*deg(\x)))},{-sin(deg(\x))*(-(3/10)*cos(3*deg(\x)) - 25/133 *cos(5*deg(\x)) + 1/3*sin(2*deg(\x)) - 3/5*sin(3*deg(\x))) + cos(deg(\x))*(-(1/6)*cos(2*deg(\x)) + 1/5*cos(3*deg(\x)) - 1/10*sin(3*deg(\x)) - 5/133*sin(5*deg(\x)))});

\draw [color=blue,fill=blue] (1.33205, 0.152523) circle(1.5pt);
\draw [color=magenta,fill=magenta] (-0.0504768, 1.22466) circle(1.5pt);
\draw [color=blue,fill=blue] (-1.12495,0.935977) circle(1.5pt);
\draw [color=magenta,fill=magenta] (-1.18102, -0.172011) circle(1.5pt);
\draw [color=blue,fill=blue] (0.183859, -1.36243) circle(1.5pt);
\draw [color=magenta,fill=magenta] (1.02573,-0.796301) circle(1.5pt);

\draw [color=blue,dash pattern={on 3pt off 2pt }] (0.5709513316,0.0374874459) -- (1.33205, 0.152523);
\draw [color=magenta,dash pattern={on 3pt off 2pt }] (-0.229640646, -0.618141762) -- (-0.0504768, 1.22466);
\draw [color=blue,dash pattern={on 3pt off 2pt }] (-0.901930125,0.760247446) --  (-1.12495,0.935977);
\draw [color=magenta,dash pattern={on 3pt off 2pt }] (0.77760233,0.393406288) -- (-1.18102, -0.172011);
\draw [color=blue,dash pattern={on 3pt off 2pt }] (0.105488466, -0.63368079) --  (0.183859, -1.36243);
\draw [color=magenta,dash pattern={on 3pt off 2pt }] (-0.44351703, 0.29974050) --(1.02573,-0.796301);

\draw[color=ccwwqq] (1,.86) node {$z$};
\draw[color=red] (.1,.47) node {$e$};

\end{tikzpicture}
\caption{The blue points are maxima of the  curvature of $z$, the magenta points are minima.
The sum of the lengths of the red arcs of the evolute $e$ equals the sum of lengths of the green arcs.}\label{fig-evolute}
\end{center}
\end{figure}
\subsection{Equilibria}
We now choose a measure $\mu$ with support in the compact convex set $K\subset \mathbb R^2$ which models the density of a distribution
of mass. The center of mass of $\mu$
is a point $O\in K$. Vice versa, given a point $O\in K$, there is a measure
supported in $K$ with center of mass $O$ (e.g.~a Dirac mass in $O$). In a 
physical model, this scenario can be realized by fixing a heavy lead ball 
in the point $O$ on a thin, lightweight plate which has shape $K$. 
If the density in $K$ is constant the center of mass is usually called the {\em centroid}.
If we allow signed measures, the center of mass can be any point $O$ in $\mathbb R^2$,
and vice versa, given an arbitrary point $O\in\mathbb R^2$, there is a signed
measure supported in $K$ with center of mass in $O$. A physical model 
can be manufactured be glueing a long, thin batten to $K$ joining $K$ to a
point $O\notin K$ and to fix a heavy lead ball at its far end in $O$.

We are interested in the following question: Suppose $K$ is equipped
with a center of mass $O$, as discussed above, and is rolling
along a horizontal straight line $\ell$. Horizontal means, that $\ell$
is perpendicular to the direction of the gravitational force $g$.
{\em What can we say about the 
number of equilibria with respect to $O$ in terms of the geometry of $\partial K$?
In particular, how many equilibrium positions are there?}

Physically, an equilibrium position is characterized by the fact, that the 
vector $v$ from the center of mass $O$ of $K$ to the contact point of $\partial K$ with the supporting 
straight line $\ell$ is parallel to the gravitational force. 
This follows from Varignon's Theorem of the resulting torque and the principle of angular momentum.
In case of a 
horizontal supporting line $\ell$, this means that $v$ is orthogonal to $\ell$.
The equilibrium is stable, if the potential energy of $K$ (i.e.~of its center of mass)
has a strict local minimum with respect to the direction $-g$, and it is unstable,
if the potential energy has a strict local maximum.
This translates into the following definition:
\begin{definition}
Let $K$ be strongly convex and compact with $\partial K$ of class $C^3$
parametrized by $\varphi\mapsto z(\varphi)=p(\varphi)u(\varphi)+p'(\varphi)u'(\varphi)$,
where the origin is chosen in the center of mass $O$ of $K$.
Then, a horizontal equilibrium position with respect to $O$ is a point $z(\varphi_0)\in\partial K$ such that $p'(\varphi_0)=0$. The equilibrium
is stable if $p$ has a strict local minimum in $\varphi_0$, and unstable if 
$p$ has a strict local maximum in $\varphi_0$.
\end{definition}
A horizontal equilibrium $z\in\partial K$ is therefore a point where the tangent at $z$ and
the line joining $z$ and the center of mass are perpendicular.
\begin{figure}[h!]
\begin{center}
\def\Epsilon{-.2}
\def\Delta{-.7}
\definecolor{ccwwqq}{rgb}{0.8,0.4,0.}
\begin{tikzpicture}[line cap=round,line join=round,x=.5cm,y=.5cm,
declare function={
        p(\t) = 4 +1.5*sin(-deg(\t + \Delta)) + .3*sin(deg(2*(-\t + \Delta))) + .3*sin(deg(3*(-\t + \Delta)));
        S1(\t) = -0.327392 + 4*\t + 0.6*cos(deg(2*(-0.7 - \t))) + 0.9*cos(deg(3*(-0.7 - \t))) + 1.5*cos(deg(0.7 + \t)) - 0.45*cos(deg(1.4 + 2*\t)) - 0.8*cos(deg(2.1 + 3*\t));
       S2(\t) = 4 + 0.3*sin(deg(2*(-0.7 - \t))) + 0.3*sin(deg(3*(-0.7 - \t))) - 1.5*sin(deg(0.7 + \t));
       null(\t)=0;
      kurve11(\t)= 9.32287 + 4*cos(deg(3.9708 - \t)) - 0.3*sin(deg(1.8708 - 4*\t)) - 0.15*sin(deg(2.5708 - 3*\t)) - 0.45*sin(deg(5.3708 +\t)) - 0.6*sin(deg(6.0708 + 2*\t));
      kurve12(\t)=3.98759 - 0.3*cos(deg(1.8708 - 4*\t)) - 0.15*cos(deg(2.5708 - 3*\t)) -  0.45*cos(deg(5.3708 +\t)) - 0.6*cos(deg(6.0708 + 2*\t)) - 4*sin(deg(3.9708 - \t));
      kurve21(\t)=0.330237 + 4*cos(deg(1.7658 - \t)) - 0.15*sin(deg(0.365797 - 3*\t)) -  0.45*sin(deg(3.1658 +\t)) - 0.6*sin(deg(3.8658 + 2*\t)) + 0.3*sin(deg(0.334203 + 4*\t));
      kurve22(\t)=3.57491 - 0.15*cos(deg(0.365797 - 3*\t)) -0.45*cos(deg(3.1658 +\t)) - 0.6*cos(deg(3.8658 + 2*\t)) - 0.3*cos(deg(0.334203 + 4*\t)) - 4*sin(deg(1.7658 -\t));
      kurve31(\t)=18.303 + 4*cos(deg(6.25898 - \t)) - 0.3*sin(deg(4.15898 - 4*\t)) -0.15*sin(deg(4.85898 - 3*\t)) - 0.45*sin(deg(7.65898 +\t)) -0.6*sin(deg(8.35898 + 2*\t);
      kurve32(\t)=4.42509 - 0.3*cos(deg(4.15898 - 4*\t)) -0.15*cos(deg(4.85898 - 3*\t])) - 0.45*cos(deg(7.65898 +\t)) -0.6*cos(deg(8.35898 + 2*\t)) - 4*sin(deg(6.25898 - \t));
      }]
\clip(-6,-1.4) rectangle (24,8.65);

\draw[black, line width=.6pt]  plot[domain=-10:50,samples=100] (\x,{null(\x)});

\draw[red,dash pattern={on 3pt off 2pt}]  plot[domain=-5:2.15,samples=100] ({S1(\x)},{S2(\x)});

\draw[red,dash pattern={on 3pt off 2pt}]  plot[domain=2.4:10,samples=100] ({S1(\x)},{S2(\x)});

\draw[ccwwqq, line width=1pt,fill=ccwwqq,fill opacity=0.25, ]  plot[domain=0:6.3,samples=100] ({kurve11(\x)},{kurve12(\x)});
\draw[ccwwqq, line width=1pt,fill=ccwwqq,fill opacity=0.25, ]  plot[domain=0:6.3,samples=100] ({kurve21(\x)},{kurve22(\x)});
\draw[ccwwqq, line width=1pt,fill=ccwwqq,fill opacity=0.25, ]  plot[domain=0:6.3,samples=100] ({kurve31(\x)},{kurve32(\x)});

\draw[darkgreen] (7.82416, 3.92522)--(7.82416, 0);
\begin{scriptsize}

\draw[magenta,line width=1pt] (7.82416, 0)--(9.61794,0);
\draw[color=magenta] (8.72105, 0) node[anchor= north] {$p'(\varphi)$};

\draw[darkgreen] (1.26851, 2.4046)--(1.26851, 0);
\draw[darkgreen] (19.2412, 5.5954)--(19.2412,0);

\draw[red] (7.82416, 3.92522)--(6.73578, 6.29019);
\draw[red] (1.26851, 2.4046)--(0.00833367, 0.126528);
\draw[red] (19.2412, 5.5954)--(21.7388, 4.86075);


\draw[red,fill=red] (7.82416, 3.92522) circle(1.5pt);
\draw[red,fill=red] (1.26851, 2.4046) circle(1.5pt);
\draw[red,fill=red] (19.2412, 5.5954) circle(1.5pt);

\draw[ccwwqq] (-.5, 6.3) node {$K$};
\draw[red] (1.26851, 2.4046) node[anchor=south] {$O$};

\draw[color=black] (7.82416, 3.8) node[anchor= east] {$\varphi$};

\draw[color=darkgreen] (7.82416, 1.9) node[anchor= west] {$p(\varphi)$};

\draw[->] plot[domain=2.00211:4.71239] ({7.82416+cos(deg(\x))},{ 3.92522+sin(deg(\x))});

\draw[black,fill=black] (9.61794, 0) circle(1.5pt);
\draw[black,fill=black] (1.26851, 0) circle(1.5pt);
\draw[black,fill=black] (19.2412, 0.) circle(1.5pt);

\draw[color=black] (23, 0) node[anchor= south] {$\ell$};
\end{scriptsize}

\end{tikzpicture}
\caption{$K$ rolling along the horizontal line $\ell$.
The dashed line is the trace of the red center of mass $O$. 
The solid red line emanating from $O$ corresponds to the angle $\varphi=0$.
Stable equilibrium on the left ($p$ has a strict local minimum), 
non-equilibrium in the middle ($p'(\varphi)\neq 0$), 
unstable equilibrium on the right ($p$ has a strict local maximum).}\label{fig-equilibrium}
\end{center}
\end{figure}
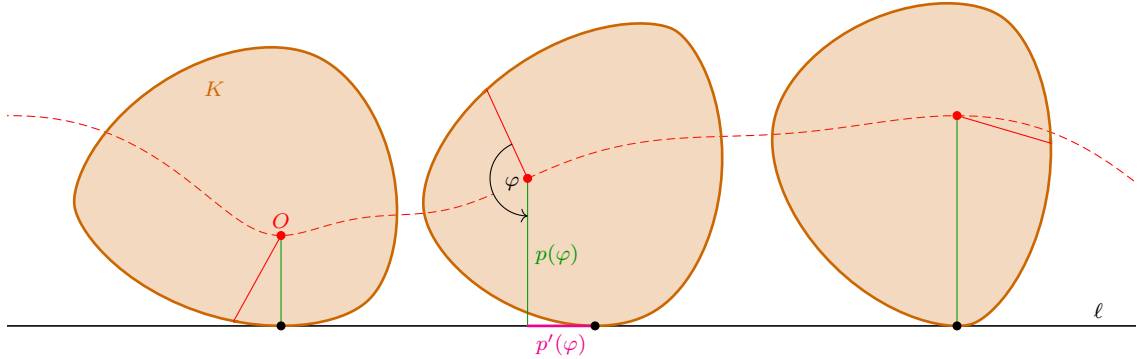
Figure~\ref{fig-equilibrium} shows a shape $K$ which has one stable
and one unstable horizontal equilibrium with respect to the center of mass $O$.
We start by investigating the number of equilibria for the special case of  the centroid $O$ of a homogeneous body.
\begin{prop}\label{number_equilibria}
Let $K$ be a convex and compact set with $C^1$ boundary. Then $K$ has at least four horizontal equilibria
with respect to its centroid.
\end{prop}
%
%
%
\begin{proof}
Suppose the boundary $\partial K$ is
given in polar coordinates as $r:S^1\to (0,\infty), \varphi\mapsto r(\varphi)$,
such that the origin is the centroid of $K$. The tangent in a point 
$z(\varphi)=r(\varphi)(\cos(\varphi),\sin(\varphi))^\top\in\partial K$ is perpendicular
to the line joining $z(\varphi)$ with the origin if and only if $r'(\varphi)=0$.
So we have to show that $r'$ has at least four zeros on $[0,2\pi)$. The condition that the centroid is at the origin leads upon integrating by parts to
\begin{equation}\label{centroidcondition}
\int_0^{2\pi}r^2(\varphi)r'(\varphi)\begin{pmatrix}\sin(\varphi)\\ -\cos(\varphi)\end{pmatrix}\,\mathrm d\varphi = 0.
\end{equation}
This implies that \eqref{centroidcondition} remains valid if $g(\varphi) = r^2(\varphi)r'(\varphi)$ is replaced by any translation $\varphi \mapsto g(\varphi-c)$, where $c\in\mathbb R$. We will now assume that there is no interval on which $r$ is constant, otherwise there is nothing to show. If $r'$ has only two zeros, then $r'>0$ on an interval of length $l\in(0,\frac\pi2]$, or 
$r'<0$ on an interval of length $l\in(0,\frac\pi2]$. We only discuss the first case (the second case is analogue). 
By a suitable translation we may assume that $r'>0$ on $(a,\pi-a)$, where $0\leq a <\pi/2$. By periodicity of $r^3$, we find
$$
\int_0^{2\pi}g(\varphi)\,\mathrm d\varphi = 0,
$$
and therefore
$$
\int_{0}^{\pi}g(\varphi)\,\mathrm d\varphi > 0.
$$
It follows that
\begin{equation}\label{integral1}\begin{aligned}
\int_0^{\pi}g(\varphi)\sin(\varphi)\,\mathrm d\varphi >\sin(a)\int_{0}^{\pi}g(\varphi)\,\mathrm d\varphi \geq0.
\end{aligned}\end{equation}
On the other hand
\begin{equation}\label{integral2}
\int_{\pi}^{2\pi}g(\varphi)\sin(\varphi)\,\mathrm d\varphi>0,
\end{equation}
and \eqref{integral1} and \eqref{integral2} contradict \eqref{centroidcondition}. Observe that the argument goes through if $r'$ has a third zero either in $(a,\pi-a)$ or in $[0,2\pi)\setminus (a,\pi-a)$ and hence we conclude that $r'$ must have at least 4 zeros as claimed.
\end{proof}
The previous proposition already appears in~\cite{Domokos1994} and could also be obtained using the Sturm-Hurwitz Theorem (Theorem 5.16 in \cite{tabachnikov}).

The next theorem reveals a connection between the number of equilibrium points of $K$
with respect to an arbitrary point $O$ which is not a point of the evolute of $\partial K$ and the winding number of the evolute of $\partial K$ around $O$.

\begin{theorem}\label{main_simple_zeros}
Let $K$ be a strongly convex compact set with $C^3$-boundary $\partial K$, and $O$ a point
in the plane. Suppose that $O$ is not a point of the evolute of $\partial K$. Then 
the number $n$ of horizontal equilibria of $K$ with respect to $O$ is given by
$$
n=2-2m,
$$
where $0\geq m\in\mathbb Z$ is the winding number of the evolute of $\partial K$ with respect to $O$.
\end{theorem}

\begin{proof} We consider the parametrisation $z(\varphi) = p(\varphi)u(\varphi) + p'(\varphi)u'(\varphi)$
of $\partial K$ with origin $O$.
The function $p$ is of class $C^3$ by Lemma \ref{regularity} and hence $p'$ is of class $C^2$. 
The evolute $e$ of $\partial K$ is then given by $e(\varphi)=p'(\varphi)u'(\varphi)-p''(\varphi)u(\varphi)$. 
In particular, since $O$ is not a point on $e$, $p'$ can only have simple zeros and by periodicity of $p$, the number $n$ of zeros of $p'$ is at least $2$. Then according to Lemma 1.1 in~\cite{zeros}, $n$ and hence the 
number of horizontal equilibria of $K$ is given by
\begin{equation}\label{numberofsimpleequilibria}
n = \frac{1}{\pi}\int_0^{2\pi}\frac{p''(\varphi)^2-p'(\varphi)p'''(\varphi)}{p'(\varphi)^2+p''(\varphi)^2}\,\mathrm d\varphi.
\end{equation}
Hence $n$ equals twice the winding number of the curve $\varphi\mapsto (p''(\varphi),p'(\varphi))$ with respect to $O$. 
The evolute can be rewritten as follows:
$$
e(\varphi) = 
\underbrace{\begin{pmatrix}\cos(\varphi) & -\sin(\varphi) \\ \sin(\varphi) & \cos(\varphi)\end{pmatrix}}_{=:R(\varphi)}\begin{pmatrix}-p''(\varphi)\\p'(\varphi)\end{pmatrix}.
$$
Since $R(\varphi)$ causes one counterclockwise rotation around the origin and since the winding number of 
$\varphi\mapsto (-p''(\varphi),p'(\varphi))$ equals $-\frac n2$, the winding number $m$ of the evolute is given by $m=1-\frac{n}{2}$.
The claim follows immediately.
\end{proof}
{\em Remark.} According to Section~\ref{const-width} the number of equilibria of a curve of constant width 
with respect to a point not on the evolute is 2 modulo 4.

The foregoing proof can be obtained by a direct computation which remains valid in a more general setting: Since $e$ is a piecewise $C^2$ immersion under the assumptions of Lemma \ref{cuspsofevolute}, the winding number of $e$ with respect to $O$ is given (see Proposition 2.3 in~\cite{residue}) by
\begin{equation}\label{windingnumberevolute}
\begin{aligned}
m  = \frac{1}{2\pi}\int_0^{2\pi}\frac{\langle Je,e'\rangle}{\|e\|^2}\,\mathrm d\varphi
= \frac{1}{2\pi}\int_0^{2\pi}\frac{p'(p'+p''')}{p'^2+p''^2}\,\mathrm d\varphi
\end{aligned}
\end{equation}
and the corresponding integrand is bounded. In the case of simple zeros of $p'$ as discussed in Theorem \ref{main_simple_zeros}, the integrand is even continuous. Then it holds that

\begin{equation}\label{directcomputation}\begin{aligned}
-2m+2 & \stackrel{\eqref{windingnumberevolute}}{=} \frac{1}{\pi}\int_0^{2\pi}\frac{-p'(p'+p''')}{p'^2+p''^2}\,\mathrm d\varphi + \frac{1}{\pi}\int_0^{2\pi}\frac{p'^2+p''^2}{p'^2+p''^2}\,\mathrm d\varphi \\
& = \frac{1}{\pi}\int_0^{2\pi}\frac{p''^2-p'p'''}{p'^2+p''^2}\,\mathrm d\varphi \stackrel{\eqref{numberofsimpleequilibria}}{=} n.
\end{aligned}\end{equation}
The last equality of this computation also holds true by Theorem 2.4 in~\cite{zeros} in a more general setting: In particular, the computation remains valid if $p'$ has zeros of order at most $2$ and the relevant integrands are continuous by the following lemma:

\begin{lemma}\label{regularityintegrand}
If $p\in C^k$, $k\geq 3$ and $p'$ only has zeros of order at most $k-1$, then the integrands in \eqref{directcomputation} are continuous.
\end{lemma}
{\em Proof.}
It suffices to show the continuity of the integrands in $0$ provided $\varphi = 0$ is a zero of $p'$ of multiplicity $k-1$. Using Proposition 2.5 in~\cite{zeros} we find by Taylor expansion
\begin{align*}
p'(\varphi) & = \left(\frac{p^{(k)}(0)}{(k-1)!}+r_0(\varphi)\right)\varphi^{k-1},\\
p''(\varphi) & = \left(\frac{p^{(k)}(0)}{(k-2)!}+r_1(\varphi)\right)\varphi^{k-2},\\
p'''(\varphi) & = \left(\frac{p^{(k)}(0)}{(k-3)!}+r_2(\varphi)\right)\varphi^{k-3},\\
\end{align*}
where $r_i$ are continous functions with $\lim_{\varphi\to 0}r_i(\varphi) = 0$. Then
$$
\lim_{\varphi\to 0}\frac{p''(\varphi)^2-p'(\varphi)p'''(\varphi)}{p'(\varphi)^2+p''(\varphi)^2} = \frac{1}{k-1}
$$
and
\makeatletter
\renewcommand\tagform@[1]{\maketag@@@{\ignorespaces#1\unskip\@@italiccorr}}
\makeatother
\begin{equation}
\lim_{\varphi\to 0}\frac{p'(\varphi)(p'(\varphi)+p'''(\varphi))}{p'(\varphi)^2+p''(\varphi)^2} = \frac{k-2}{k-1}.\tag{$\Box$}
\end{equation}
\makeatletter
\renewcommand\tagform@[1]{\maketag@@@{\ignorespaces(#1)\unskip\@@italiccorr}}
\makeatother

In order to prove Theorem~\ref{main_arbitrary_zeros}
it remains to discuss the cases where the center of mass $O$ of $K$ is 
possibly a point of the evolute. We continue to assume, as in 
Lemma \ref{cuspsofevolute}, that the radius of 
curvature of $\partial K$ has only finitely many stationary 
points and $\partial K$ is of class $C^3$. We will distinguish two cases:


\begin{enumerate}
\item If $O$ is a \emph{regular point} of the evolute of $\partial K$, then whenever $e(\varphi_0)=0$, it holds that $e'(\varphi_0)\ne0$. 
This corresponds to the two black points in Figure~\ref{fig:example} which are labeled by 3 and 4.
Since $e'=-\rho'u$ this means that $\varphi_0$ is not a stationary point of $\rho$ and hence $p'''(\varphi)\ne 0$. Therefore the set $e^{-1}(0)$ consists of zeros of $p'$ of multiplicity $2$ and we conclude that $p'$ has zeros of order at most $2$.


In this case, computation~\eqref{directcomputation} remains valid by Proposition 2.3 in~\cite{residue} and Theorem 2.4 in~\cite{zeros} and the integrands are continuous according to Lemma~\ref{regularityintegrand}. Proposition 2.2 in~\cite{residue} tells us (since the angles in $O$ are equal to $\pi$) that $2m\in\mathbb Z$ and we conclude that $n=2-2m$, but $m$ might be half-integer valued.

\item If $O$ is a \emph{singular point} of the evolute of $\partial K$, there exist values $\varphi_0$ such that $e(\varphi_0)=e'(\varphi_0)=0$.
See, e.g., the  black point in Figure~\ref{fig:example} which is labeled by 2.
 In this case, $\varphi_0$ is a stationary point of $\rho$ which is either a saddle point or a cusp of $e$ according to Lemma~\ref{cuspsofevolute}. Since $e=p'u'-p''u$ and $e'=-(p'+p''')u$ we conclude that such points are zeros of $p'$ of order at least 3.

The computation \eqref{directcomputation} remains valid in this case if we can show that $p'$ is an \emph{admissible} function in the sense of Definition 2.4 in~\cite{zeros}. More precisely, the first equality is then justified by Proposition 2.3 in~\cite{residue} and the last one by Theorem 2.4 in~\cite{zeros}. According to Proposition 2.2 in~\cite{residue} (since the angles in $O$ are $0$, $\pi$ or $2\pi$) we find again $2m\in\mathbb Z$.

Since $p'\in C^2$, it suffices to show that the zeros of $p'$ are \emph{admissible} in the sense of Definition 2.1 in~\cite{zeros}, i.e.\ we have to show that whenever $p'(\varphi_0)=0$, then
$$
\lim_{\varphi\nearrow \varphi_0}\frac{p''(\varphi)}{p'(\varphi)}=-\infty \text{ and }\lim_{\varphi\searrow \varphi_0}\frac{p''(\varphi)}{p'(\varphi)}=+\infty.
$$

Let $\varphi_0=0$ be a zero of $p'$. If this zero is of multiplicity one or two, then the admissibility follows immediately from the 5th point in the Remark after Definition 2.1 in~\cite{zeros}. In the present case we assume that $p'(0)=p''(0)=p'''(0)=0$. In this case, $\rho(0)=p(0)$ and $\rho'(0)=0$ and we can solve the ODE $\rho=p+p''$ with  initial value $p'(0)=0$ in order to obtain
$$
p(\varphi)= \int_0^\varphi \sin(\varphi-t)\rho(t)\,\mathrm dt+p(0)\cos(\varphi)=\int_0^\varphi \sin(\varphi-t)(\rho(t)-\rho(0))\,\mathrm dt+\rho(0).$$
Upon integrating by parts (since $\rho$ is of class $C^1$) we get the formulas
$$\begin{aligned}
p(\varphi) & = \rho(\varphi) - \int_0^{\varphi}\cos(\varphi-t)\rho'(t)\,\mathrm dt,\\
p'(\varphi) & = \int_0^{\varphi}\sin(\varphi-t)\rho'(t)\,\mathrm dt,\\
p''(\varphi) & = \int_0^{\varphi}\cos(\varphi-t)\rho'(t)\,\mathrm dt. \end{aligned}$$

Since the number of zeros of $\rho'$ is finite, we can consider the case where e.g.\ $\rho' > 0$ on $(0,\varphi)$ provided $\varphi>0$ is small enough. Then

$$p''(\varphi)  \geq \int_0^\varphi (1-(\varphi-t)^2)\rho'(t)\,\mathrm dt\geq (1-\varphi^2)\int_0^{\varphi}\rho'(t)\,\mathrm dt$$
and
$$
p'(\varphi) \leq \varphi \int_0^{\varphi}\rho'(t)\,\mathrm dt.
$$
We conclude that
$$
\frac{p''(\varphi)}{p'(\varphi)}\geq \frac{1-\varphi^2}{\varphi}\stackrel{\varphi \searrow 0}{\to}+\infty.
$$
The remaining cases are similar and we find
$$
\lim_{\varphi\nearrow 0}\frac{p''(\varphi)}{p'(\varphi)}=-\infty\text{ and }\lim_{\varphi\searrow 0}\frac{p''(\varphi)}{p'(\varphi)}=+\infty
$$
and therefore, the zeros of $p'$ are admissible.
\end{enumerate}
This concludes the proof of Theorem~\ref{main_arbitrary_zeros}.


\emph{Remarks.}
\begin{enumerate}
\item According to Section~\ref{const-width} the number of equilibria of a curve of constant width with respect to a point on the evolute is even.

\item It follows from the conclusion of Theorem \ref{main_arbitrary_zeros}, that the number of zeros of $p'$ is finite. This also follows a priori from the fact that the number of extrema of $\rho$ is finite. Indeed, if $p'(\varphi_0)=0$, then the tangent of $z$ in $\varphi_0$ is perpendicular to $z(\varphi_0)$ and $z(\varphi_0)$ is parallel to $e'(\varphi_0)$. Since every arc of the evolute $e$ is convex, there are at most 2 tangents to such an arc through $z(\varphi_0)$. 
Since the curvature of $\partial K$ has only finitely many stationary points, $e$ is made of only finitely many arcs and there are only twice as many zeros of $p'$ as there are extrema of $\rho$.
\item For points on the evolute, one can formulate the result alternatively as follows: If the center of mass $O$ lies on
the evolute, then the number of equilibrium positions with respect to $O$ is the
average of the number of equilibrium positions in the neighbouring areas defined by the evolute,
where each neighbouring area is weighted by its angle in $O$. For example, the number
of equilibrium positions in the black points in Figure~\ref{fig:example} can be obtained
in this way: The 3 is the average of 2 and 4, the 4 is the the average of 
4 and 4 (with equal weight) and 2 and 6 (with equal weight), and the 2 is
the average of 2 (with full weight) and 4 (with weight zero).
\end{enumerate}

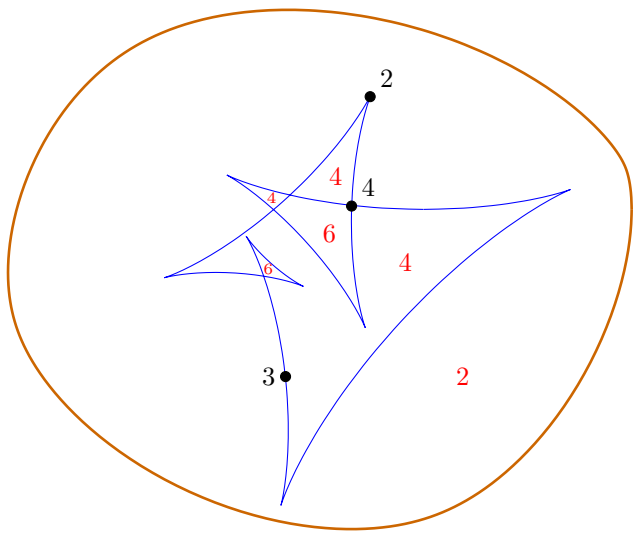
\begin{figure}[h!]
\begin{center}
\definecolor{ccwwqq}{rgb}{0.8,0.4,0.}
\begin{tikzpicture}[scale=3.8,domain=0:6.28319]
\draw[line width=1pt,smooth,samples=100,color=ccwwqq,domain=0:6.28319] plot ({-0.1 + 1.13636*cos(\x r) - 0.0454545*cos(3*\x r) + 
 0.1*cos(\x r)*sin(\x r) + 0.0689655*sin(3*\x r) - 0.025*sin(4*\x r) - 
 0.0413793*sin(5*\x r)},{-0.1 + 0.05*cos(2*\x r) + 0.0689655*cos(3*\x r) + 0.025*cos(4*\x r) + 
 0.0413793*cos(5*\x r) + 0.863636*sin(\x r) - 0.0454545*sin(3*\x r)});
 \draw[smooth,samples=100,color=blue,domain=0:6.28319] plot ({-0.1 + 0.272727*cos(\x r) + 0.0909091*cos(3*\x r) + 
 0.3*cos(\x r)*sin(\x r) + 
 0.275862*sin(3*\x r) + 0.075*sin(4*\x r) + 
 0.165517*sin(5*\x r)},{-0.1 + 0.15*cos(2*\x r) + 0.275862*cos(3*\x r) - 
 0.075*cos(4*\x r) - 0.165517*cos(5*\x r) - 
 0.272727*sin(\x r) + 0.0909091*sin(3*\x r)});
 \draw[color=red] (-.066,0) node {\small$6$};
 \draw[color=red] (.4,-.5) node {\small$2$};
 \draw[color=red] (.2,-.1) node {\small$4$};
 \draw[color=red] (-.044,.2) node {\small$4$};
 \draw[color=red] (-.269,.124) node {\tiny$4$};
 \draw[color=red] (-.28,-.124) node {\tiny$6$};

\draw [fill=black] (.011,.097) circle (.5pt)  node[anchor=south west] {\small$4$};
\draw [fill=black] (.076,.48) circle (.5pt)  node[anchor=south west] {\small$2$};
\draw [fill=black] (-.22,-.5) circle (.5pt)  node[anchor=east] {\small$3$};

\end{tikzpicture}
\caption{The number of equilibria with respect to a given center of mass
is 2 minus twice the winding number of the evolute. The number of
equilibria is indicated in the figure for areas bounded by the evolute in red, 
and for selected black points on the evolute.}\label{fig:example}
\end{center}
\end{figure}

\section{Oblique equilibria}\label{sectionoblique}
Here we investigate the equilibrium positons of $K$ with respect to
a center of mass $O$ on an oblique line $\ell$ with angle of inclination $\alpha\neq 0$.
The situation is shown in Figure~\ref{fig:oblique}. We can immediately
read off the condition for an equilibrium position in terms of the support function $p$:
An equilibrium point is characterised by the condition
\begin{equation}\label{eq:oblique}
p'(\varphi) = \tan(\alpha)p(\varphi),
\end{equation}
or, if $p(\varphi)\neq 0$, equivalently by
$$
\tan(\alpha)=\frac{p'(\varphi)}{p(\varphi)}=(\ln |p(\varphi)|)'.
$$
In particular, the number $n_\alpha$ of solutions of~(\ref{eq:oblique}) on $[0,2\pi)$ corresponds to the
number of equilibrium points. This number varies with $\alpha$: See Figure~\ref{fig:equi}. 
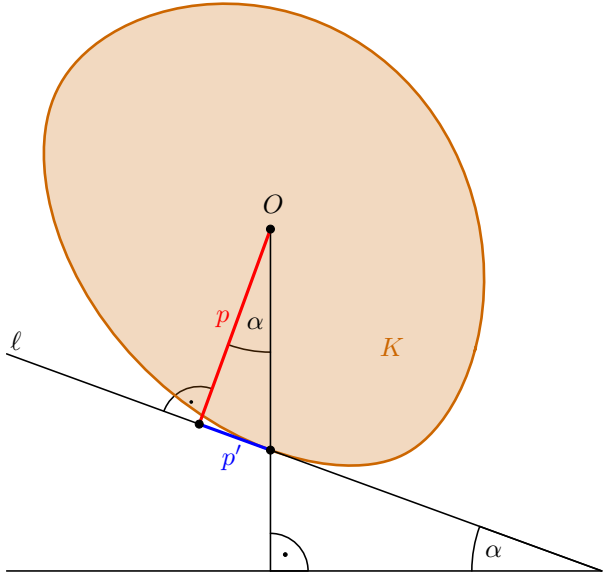
\begin{figure}[h!]
\begin{center}
\definecolor{ccwwqq}{rgb}{0.8,0.4,0.}
\begin{tikzpicture}[line cap=round,line join=round,x=1.0cm,y=1.0cm,
declare function={
f1(\t) = 14.3828+cos(deg(6.54781-\t))*(3.0+0.01*cos(deg(1.0+\t))+0.3*cos(deg(2.0*(1.0+\t)))+0.1*cos(deg(3.0*(1.0+\t)))-0.01*sin(deg(1.0+\t))+0.3*sin(deg(2.0*(1.0+\t))))+sin(deg(6.54781-\t))*(-0.01*cos(deg(1.0+\t))+0.6*cos(deg(2.0*(1.0+\t)))-0.01*sin(deg(1.0+\t))-0.6*sin(deg(2.0*(1.0+\t)))-0.3*sin(deg(3.0*(1.0+\t))));
f2(\t) = -2.30993-sin(deg(6.54781-\t))*(3.0+0.01*cos(deg(1.0+\t))+0.3*cos(deg(2.0*(1.0+\t)))+0.1*cos(deg(3.0*(1.0+\t)))-0.01*sin(deg(1.0+\t))+0.3*sin(deg(2.0*(1.0+\t))))+cos(deg(6.54781-\t))*(-0.01*cos(deg(1.0+\t))+0.6*cos(deg(2.0*(1.0+\t)))-0.01*sin(deg(1.0+\t))-0.6*sin(deg(2.0*(1.0+\t)))-0.3*sin(deg(3.0*(1.0+\t))));
}]
\clip(10.874812424743041,-7.0478920695722) rectangle (18.93098571027498,0.935539244984225);

\draw [shift={(18.787449918128235,-6.857954565316334)},line width=.6pt] (0,0) -- (159.94631682876312:1.7276218106671206) arc (159.94631682876312:180.:1.7276218106671206) ; 
\draw [shift={(14.3828,-2.30993)},line width=.6pt] (0,0) -- (-110.05352282957887:1.6366943469477984) arc (-110.05352282957887:-89.9999817003804:1.6366943469477984) ; 
\draw [shift={(14.3828,-6.857954565316333)},line width=.6pt] (0,0) -- (0.:.5) arc (0.:89.9999665361928:.5) ; 
\draw [shift={(13.435734986778147,-4.904426119164102)},line width=.6pt] (0,0) -- (69.94647717042115:0.5) arc (69.94647717042115:159.94647717042122:0.5) ; 
\draw[ccwwqq, line width=1pt,fill=ccwwqq,fill opacity=0.25, ]  plot[domain=0:6.3,samples=100] ({f1(\x)},{f2(\x)});

\draw [line width=.6pt,domain=10.874812424743041:18.787449918128235] plot(\x,{(-0.--6.857954565316334*\x)/-18.787449918128235});
\draw [line width=.6pt,domain=10.874812424743041:18.787449918128235] plot(\x,{(--128.60141327971814-0.*\x)/-18.752153000562526});
\draw [line width=.6pt] (14.3828,-2.30993)-- (14.3828,-6.857954565316333);
\draw [line width=1.2pt,color=red] (14.3828,-2.30993)-- (13.435734986778147,-4.904426119164102);
\fill[line width=.6pt] (13.33,-4.61) circle (0.7pt);
\draw [line width=1.2pt,color=blue] (13.435734986778147,-4.904426119164102)-- (14.382800939062637,-5.2501182094154135);
\begin{small}
\draw[color=ccwwqq] (16,-3.8745235857678804) node {$K$};
\draw [fill=black] (14.3828,-2.30993) circle (1.5pt);

\draw [fill=black] (14.58,-6.64) circle (.7pt);

\draw[color=black] (14.420983509796605,-1.98) node {$O$};
\draw[color=black] (11,-3.8) node {$\ell$};
\draw [fill=black] (13.435734986778147,-4.904426119164102) circle (1.5pt);
\draw[color=red] (13.75,-3.5) node {$p$};
\draw [fill=black] (14.382800939062637,-5.2501182094154135) circle (1.5pt);
\draw[color=black] (17.35,-6.6) node {$\alpha$};
\draw[color=black] (14.18,-3.56) node {$\alpha$};
\draw[color=blue] (13.875418727480671,-5.35) node {$p'$};
\end{small}
\end{tikzpicture}
\caption{Equilibrium position on an oblique line $\ell$.}\label{fig:oblique}
\end{center}
\end{figure}

If we denote by $v=\begin{pmatrix}\textcolor{white}{-}\cos(\alpha)\\-\sin(\alpha)\end{pmatrix}$
the vector in the downhill direction of $\ell$ and by $s(\varphi)$ the arclength on
$\partial K$ corresonding to the parameter interval $[0,\varphi]$ we can express
the position of $O$ in coordinates with respect to fixed horizontal and vertical axis as
$$
O(\varphi)=\begin{pmatrix}O_1(\varphi)\\O_2(\varphi)\end{pmatrix}=(s(\varphi)-p'(\varphi))v+ p(\varphi)v^\bot,
$$
where $v^\bot=\begin{pmatrix}\sin(\alpha)\\\cos(\alpha)\end{pmatrix}$. An equilibrium
corresponds  to a point with stationary potential energy, i.e., $O_2'(\varphi)=0$. 
A sufficient condition for an equilibrium to be stable is
 $O_2''(\varphi)>0$, corresponding to a strict local minimum of the 
potential energy. Similarly,  $O_2''(\varphi)<0$ implies that an equilibrium is unstable. According to~(\ref{eq-arc}) we have
$$
O_2'(\varphi)=p'(\varphi)\cos(\alpha)-p(\varphi)\sin(\alpha),\text{ and }
O_2''(\varphi)=p''(\varphi)\cos(\alpha)-p'(\varphi)\sin(\alpha).
$$
Thus, for an equilibrium $O_2'(\varphi)=0$, we obtain
\begin{itemize}
\item  if $p''(\varphi)>p(\varphi)\tan^2(\alpha)$, then $\varphi$ is a stable equilibrium,
\item  if $p''(\varphi)<p(\varphi)\tan^2(\alpha)$, $\varphi$ is an unstable equilibrium.
\end{itemize}
In particular a center of mass $O$ on $\partial K$ is always a stable equilibrium.
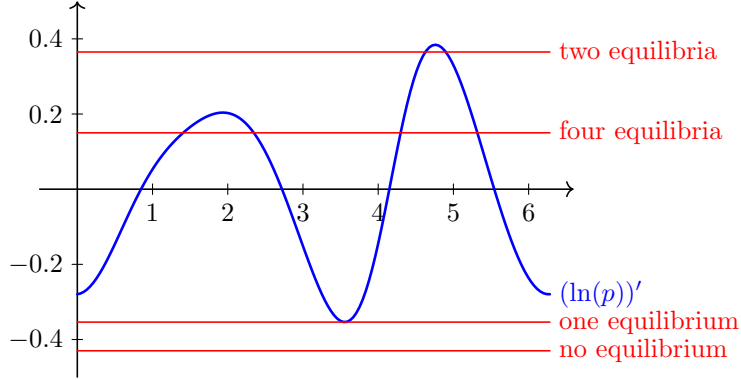
\begin{figure}[h!]
\begin{center}
\begin{tikzpicture}[line cap=round,line join=round,x=1.0cm,y=5.0cm]

\draw[->, line width=.6pt] (-.5,0)--(6.6,0);
\draw[->, line width=.6pt] (0,-.5)--(0,.5);

\begin{small}
\foreach \x/\xtext in {1/1, 2/2, 3/3, 4/4, 5/5, 6/6}
    \draw[shift={(\x,0)}] (0pt,2pt) -- (0pt,-2pt) node[below] {$\xtext$};
    
\foreach \y/\ytext in { -0.2/-0.2, -0.4/-0.4, 0.2/0.2, 0.4/0.4}
    \draw[shift={(0,\y)}] (2pt,0pt) -- (-2pt,0pt) node[left] {$\ytext$};

\draw [color=blue, line width=1pt,domain=-0:6.28319, samples=150] plot (\x, {(-0.01*cos((1 + \x) r) + 0.6*cos((2*(1 + \x)) r) - 0.01*sin((1 + \x) r) - 
      0.6*sin((2*(1 + \x)) r) - 0.3*sin((3*(1 + \x)) r))/(3 + 
   0.01*cos((1 + \x) r) + 
      0.3*cos((2*(1 + \x))r) + 0.1*cos((3*(1 + \x)) r) - 0.01*sin((1 + \x) r) + 
   0.3*sin((2*(1 + \x)) r))});
   
\draw [color=red, line width=.6pt,domain=-0:6.28319, samples=5] plot (\x, {tan(.35 r)});   
\draw [color=red, line width=.6pt,domain=-0:6.28319, samples=5] plot (\x, {.15});   
\draw [color=red, line width=.6pt,domain=-0:6.28319, samples=5] plot (\x, {-0.353677});   
\draw [color=red, line width=.6pt,domain=-0:6.28319, samples=5] plot (\x, {-.43});

\draw[color=blue] (6.28319,-0.27952) node[anchor=west] {$(\ln(p))'$};
\draw[color=red] (6.28319,0.365) node[anchor=west] {two equilibria};
\draw[color=red] (6.28319,0.15) node[anchor=west] {four equilibria};
\draw[color=red] (6.28319,-0.353677) node[anchor=west] {one equilibrium};
\draw[color=red] (6.28319,-0.43) node[anchor=west] {no equilibrium};

\end{small}
\end{tikzpicture}
\caption{Number of equilibrium points with respect to $\alpha$. The value 
$\tan(\alpha)$ is drawn for several values of $\alpha$.
}\label{fig:equi}
\end{center}
\end{figure}

Interesting observations are
\begin{prop}\label{obliqueprop}
\begin{enumerate}
\item There are shapes $K$ which have oblique equilibrium points with respect to the centroid for angle of inclination $\alpha$, 
but no equilibrium for angle $-\alpha$.
\item For all $\alpha\in(-\pi/2,\pi/2)$ there exist shapes $K$ which have stable equilibrium positions
with respect to $\alpha$ for the centroid.
\item For all small $\alpha> 0$ there exist shapes $K$ which have only one metastable equilibrium, and no other equilibrium,
with respect to $\alpha$ for the centroid.
\item For all small $\alpha> 0$ there exist shapes $K$ which have only one stable and one unstable equilibrium
with respect to $\alpha$ for the centroid. 
\end{enumerate}
\end{prop}
 {\em Remark.} The last two properties are in sharp contrast to Proposition~\ref{number_equilibria} for $\alpha=0$.
\begin{proof}
Consider the support function $$p(\varphi)=3 -\tfrac{279}{8570} \sin(\varphi) + \tfrac{36}{857} \cos(\varphi)+ \tfrac3{10} \bigl(\sin(2 \varphi) + \cos(2 \varphi)\bigr) + \tfrac1{5}\cos(3 \varphi) .$$
One can check, that $p+p''>0$ and that $\max (\ln p)'+\min (\ln p)'>0$. Moreover, the centroid is at
the origin. So, for $\alpha$ such that
$-\min (\ln p)'<\tan(\alpha)<\max (\ln p)'$ the shape with this support function $p$ has the
property mentioned in the first part of the proposition.

For the second part, observe that the ellipse with half axis $a>1$ and $b=1$ has two
stable and two unstable equilibria with respect to its center for every angle $\alpha<\arctan(\frac{a^2-1}{2a})$.

\textcolor{black}{For the rest, let $$p_c(\varphi) = 3 + 3c (\cos(2 \varphi) +\sin(2 \varphi)) + 2c\cos(3 \varphi) + \tfrac{36 c^2}{9 - 43 c^2}
   \cos(\varphi) - \tfrac{9 (4 - 9 c) c^2}{9 - 43 c^2} \sin(\varphi).$$
Let $c>0$ be sufficiently small, so that $z=p_cu+p_c'u'$ parametrizes the boundary of a convex body $K$. By construction, the centroid of $K$ lies at the origin. One can check that the function $p_c'(\varphi)/p_c(\varphi)$ has a unique maximum for each such $c$. Choose $\alpha_c$ in such a way that $\tan(\alpha_c) = \max p'_c/p_c$. Then $K$ has exactly one equilibrium for $\alpha_c$ and for a slightly smaller angle one stable and one unstable equilibrium. Since
$\frac{p'_c}{p_c}$ converges uniformly  to $0$ for $c\searrow 0$ the claim follows.}
\end{proof}



In view of Theorem~\ref{main_arbitrary_zeros}
it is natural to ask, if the number $n_\alpha$ of oblique equilibria with respect to 
angle $\alpha>0$ can be obtained as $n_{\alpha}=2-2m_{\alpha}$, where $m_\alpha$ is the winding number of the evolute of a suitable modification of $\partial K$. Consider therefore again a strongly convex and compact set $K$ with $C^3$ boundary and such that the radius of curvature of $\partial K$ has only finitely many stationary points. Let $z=pu+p'u'$ be the usual $C^2$ parametrization of $\partial K$ and let $e=p'u'-p''u$ be the evolute of $\partial K$. Define $e_{\alpha}=e - \tan(\alpha)Jz$ and $p'_\alpha = p'-\tan(\alpha)p$. Let $p_\alpha$ be a primitive of $p'_\alpha$ with constant of integration large enough such that $p_\alpha + p''_\alpha =:\rho_\alpha>0$. In this case, $p_\alpha$ is again the support function of a curve $C_\alpha$ and the evolute of $C_\alpha$ is precisely $e_\alpha$.

\begin{prop}
If the curvature of $C_\alpha$ admits only finitely many stationary points, then the number $n_\alpha$ of oblique equilibria of $\partial K$ with respect to $O\in\R^2$ and
angle of inclination  $\alpha$ is given by
$n_\alpha = 2 - 2m_\alpha$,
where $m_\alpha\in\frac12\mathbb Z$ is the winding number of the evolute of $C_\alpha$ with respect to $O$.
\end{prop}

\begin{proof}
Observe that $e_\alpha$ is a piecewise $C^2$ immersion, since $e$ is piecewise $C^2$, $z$ is of class $C^2$ and the number of zeros of $e_\alpha'$ is finite. In this case, the winding number $m_\alpha$ of $e_\alpha$ is given by
$$m_\alpha  = \frac{1}{2\pi}\int_0^{2\pi}\frac{\langle Je_\alpha,e_\alpha'\rangle}{\|e_\alpha\|^2}\,\mathrm d\varphi$$
and using $p'_{\alpha}=p'-\tan(\alpha)p$ we obtain
\begin{equation}\label{directcomputation:oblique}
-2m_\alpha + 2 = \frac{1}{\pi}\int_0^{2\pi}\frac{{p''_{\alpha}}^2-{p'_{\alpha}} p'''_{\alpha}}{{p'_{\alpha}}^2+{p''_{\alpha}}^2}\,\mathrm d\varphi=n_\alpha,
\end{equation}
in analogous manner to the case where $\alpha=0$.\end{proof}

{\em Remark.}
It is clear by definition that $e_\alpha$ diverges as $\alpha\to\pm\frac\pi2$, however, the renormalized perturbed evolute $e_\alpha/\tan(\alpha)$ converges to $-Jz$ as $\alpha\to\pm\frac\pi2$. Moreover $m_\alpha\to1$ as $\alpha\to \pm\frac{\pi}{2}$ so that $n_{\pm\frac\pi2}=0$, as expected.

\bibliographystyle{plain}
\bibliography{gomboc}
\end{document}